\documentclass[a4paper, 12pt, final]{ectaart} 
\usepackage{amsmath}
\usepackage{amssymb}
\usepackage{bm}
\usepackage{amsthm}
\usepackage{graphicx}
\usepackage[round,sectionbib]{natbib}
\usepackage{setspace} 
\usepackage{a4wide}
\usepackage{url}
\usepackage{caption}
\usepackage{multirow}
\usepackage[colorlinks]{hyperref}
\input{ee.sty}

\newtheorem{prp}{Proposition}
\newtheorem{lem}{Lemma}

\newtheorem{con}{Condition}
\newtheorem{exm}{Example}

\DeclareMathOperator*{\argmin}{arg\,min}
\DeclareMathOperator*{\argmax}{arg\,max}

\DeclareMathOperator*{\rank}{rank}

\allowdisplaybreaks
\begin{document}
	\begin{frontmatter}
		\title{Directing Power Towards Conic Parameter Subspaces}
		\runtitle{Directing Power Towards Conic Parameter Subspaces}
		\vspace{.5cm}
		PRELIMINARY VERSION \\
		\vspace{.5cm}
		12 November 2019
		
		\begin{aug}
			\author{\fnms{Nick} \snm{Koning}\ead[label=e1]{n.w.koning@rug.nl}}
			
			
			\address{N.W. Koning, Faculty of Economics and Business, University of Groningen, PO Box 800, 9700 AV Groningen, The Netherlands. e-mail: \printead{e1}.}
			
		\end{aug}
		
		\begin{abstract}
			
			\noindent For a high-dimensional parameter of interest, tests based on quadratic statistics are known to have low power against subsets of the parameter space (henceforth, parameter subspaces). In addition, they typically involve an inverse covariance matrix which is difficult to estimate in high-dimensional settings. I simultaneously address these two issues by proposing a novel test statistic that is large in a conic parameter subspace of interest. This test statistic generalizes the Wald statistic and nests many well-known test statistics. For a given parameter subspace, the statistic is free of tuning parameters and suitable for high-dimensional settings if the subspace is sufficiently small. It can be computed using regularized linear regression, where the type of regularization and the regularization parameters are completely determined by the parameter subspace of interest. I illustrate the statistic on subspaces that consist of sparse or nearly-sparse vectors, for which the computation corresponds to $\ell_0$- and $\ell_1$-regularized regression, respectively.
		\end{abstract}
		
		\begin{keyword}
			High-dimensional testing, quadratic statistic, Wald statistic, sparse alternatives, power enhancement, best subset selection, lasso, regularization
		\end{keyword}
	\end{frontmatter}

	\section{Introduction}
		
		
	
		This paper concerns testing hypotheses on the location of a high-dimensional parameter $\vmu$, whose dimension $p$ may exceed the number of observations $n$. Traditionally, such tests are often based on a quadratic statistic like the Wald statistic. However, tests based on quadratic statistics are known to have low power against subsets of the parameter space (henceforth, parameter subspaces) in high dimensions. This has provoked a recent interest in constructing tests that direct power towards parameter subspaces of interest (see e.g. \citealp{fan2015power, kock2019power}).
		
		Such parameter subspaces arise if the econometrician has some prior information about the nature of violations of the null hypothesis. For example, suppose that we want to test the null hypothesis $H_0 : \vmu = \vzeros$, describing a collection of $p$ equalities. Then one may have prior information that if $H_0$ is false, only some number $k < p$ of these equalities are violated. That is, the parameter $\vmu$ is expected to be $k$-sparse under the alternative (i.e. containing $p-k$ zeros). So, we may want to use a test that directs power towards the parameter subspace consisting of $k$-sparse vectors. The following example describes how such situations can arise in practice.
		\begin{exm}
			Multi-factor pricing models describe the excess returns of assets as a linear combination of a limited number of factors. In particular, for a collection of $p$ assets at time $n$, let $\vy_n$ be the $p$-vector of excess returns, let $\vf_n$ be the $m$-vector containing the $m$ observable factors and let $\vepsi_n$ be a $p$-vector of errors. Let $\mB$ denote the $p \times m$ matrix of factor loadings, and $\vmu$ the $p$-vector of asset-specific returns that are not captured by the factors (also known as `alpha' in the finance literature). Then, a multi-factor pricing model can be written as
			\begin{align*}
				\vy_n = \vmu + \mB\vf_n + \vepsi_n.
			\end{align*}
			If the factor model is well specified, efficient pricing theory suggests that the excess returns should be fully captured by the factors so that $\vmu = \vzeros$. To assess this theory, one may want to test the null hypothesis $H_0 : \vmu = \vzeros$. As such models are well founded in theory, one may expect that if this hypothesis is false then it would only be violated by a small number of exceptional assets. So, if the alternative hypothesis holds, then the underlying vector of excess returns is likely to be sparse. Therefore, one may want to use a test that directs power towards a parameter subspace that consists of sparse vectors.
		\end{exm}
		
		Besides low power against parameter subspaces, another problem with quadratic statistics is that they typically involve and inverse covariance matrix,  which is usually not known in practice and must be estimated. This is problematic in high-dimensional settings, as standard estimators of such a covariance matrix are not invertible if $p > n$. Without imposing additional assumptions, this means that the Wald statistic does not exist. The solution provided in the literature is to estimate the covariance matrix by imposing restrictions on its structure, through regularization or other means (see e.g. \citealp{srivastava2008test, chen2011regularized}). However, if no information about the covariance matrix is available, then it is not clear which restrictions should be used. Imposing the wrong restrictions may lead to poor estimates, which in turn may lead to a loss of power.
		
		I simultaneously address these two issues by proposing a novel test statistic that is large in a conic parameter subspace that is specified by the econometrician.\footnote{A cone $\calC$ is a collection of vectors that is closed under positive scalar multiplication. That is, if $\vlambda \in \calC$, then $\gamma\vlambda \in \calC$, for all $\gamma > 0$.} This test statistic is a generalization of the Wald statistic. It nests many well-known test statistics that can be recovered by choosing a particular parameter subspace of interest. 
		
		For the test statistic to exist, a restricted eigenvalue condition must be satisfied. The strength of this condition depends directly on the parameter subspace of interest. The condition is weaker than the condition that the sample covariance matrix is positive-definite. Unlike the standard Wald statistic, the test statistic can therefore be used even in high-dimensional settings, as long as the parameter subspace of interest is sufficiently `small'.
		
		In addition, I show that the computation of the test statistic can be formulated as a quadratic minimization problem, where the arguments over which is minimized are restricted to the parameter subspace of interest. If the standard sample covariance matrix and sample mean are used as estimators, this problem reduces to regularized linear regression with a constant dependent variable. I provide an additional result for the special case of diagonal estimators of the covariance matrix and parameter subspaces defined by sign or sparsity restrictions. In that case, the computation reduces to solving a minimum distance problem which can lead to a closed-form test statistic.

		The test statistic is illustrated with the parameter subspace that consists of all $k$-sparse vectors. In case of a $k$-sparse parameter subspace, the existence criterion reduces to the requirement that the covariance estimator is at least of rank $k$. This condition is typically satisfied if $k < n$. The computation of the test statistic is a special case of best-subset selection, also known as $\ell_0$-regularization. Best-subset selection has seen a recent surge of interest after \citet{bertsimas2016best} showed that this problem is easier to solve than was commonly assumed, by exploiting gradient descent methods and commercially available mixed-integer optimization solvers.\footnote{See e.g. \citet{hastie2017extended,mazumder2017subset,hazimeh2018fast,koning2018sparse} for more recent work.} If one uses a diagonal covariance matrix estimator, then the test statistic has a closed form and is related to threshold-type statistics (see e.g., \citealp{fan1996test, zhong2013tests}). 
		
		In addition, I recover the parameter subspace for which the computation of the test statistic coincides with Lasso, which is another popular technique used to obtain sparse solutions \citep{tibshirani1996regression}. This parameter subspace consists of a union of convex cones that are centered around sparse vectors, whose widths depend on the Lasso regularization parameter. Therefore, the parameter subspace corresponding to Lasso can be interpreted as the collection of vectors that are `close' to a sparse vector. Intuitively, these may be used if one suspects only a limited number of large violations of the null hypothesis $H_0 : \vmu = \vzeros$ might occur.
		
		For the illustration of the test statistic with sparse and nearly-sparse cones, I construct a critical value using reflection-based randomization (see e.g. \citealp{lehmann2006testing}). While this construction relies on a symmetry assumption on the distribution of the error term, the advantage is that size control is guaranteed in small samples, even if $p \gg n$.
	
		The power properties of the tests proposed in this paper are studied using Monte Carlo simulation. The results suggest that the tests perform favorably compared to the Wald test if $p < n$ and retains power even if $p \gg n$. In addition, the test typically outperforms the power enhancement technique of \citet{fan2015power}.

		The power enhancement technique is a recently proposed method that increases power in a parameter subspaces of interest. This technique works as follows: given some initial test, one chooses an enhancement test which has asymptotic size zero and is consistent against a parameter subspace, against which the initial test is not consistent. The combination of these two tests has asymptotic size equal to the initial test, and may be consistent against a strictly larger set of alternatives. The existence of power enhancement tests has recently been discussed by \citet{kock2019power}.
			
		The main advantage of the methodology proposed in this paper over the power enhancement technique is that it is constructive: when the econometrician specifies a parameter subspace of interest, this simultaneously defines a test statistic. For the power enhancement technique, both an initial and enhancement test must be chosen by the econometrician. A typical choice for the initial test, as used by \citet{fan2015power}, would be based on a Wald-type statistic, which suffers from the issues in high-dimensional settings that were described above. Finding an appropriate enhancement test for a given parameter subspace may also not be straightforward in practice.
		
		The advantage of the power enhancement technique is its attractive asymptotic properties, which promise a strictly larger subset of the parameter space against which the test is consistent while maintaining asymptotic size. It remains unclear how this translates to finite samples. \citet{fan2015power} conduct several numerical experiments that suggest small size distortions may occur in finite samples. Larger size distortions are found in the Monte Carlo simulations presented in Section \ref{sec:mc}. In order to compensate for the size distortion, one could reduce the size of the initial test. This would effectively result in a test that directs power towards the subspace of interest by taking power away from the complement of this subspace.

		\subsection{Notation}
			For the remainder of the paper, the following notation is used. For a set $\calA$, the set $\calA^0 = \calA \cup \{\vzeros\}$ denotes the union of the set and the origin. The set $\calA^{\emptyset} = \calA \setminus \{\vzeros\}$ excludes the origin. Let $\mI_n = (\ve_1, \ve_2, \dots, \ve_n)$ denote the $n \times n$ identity matrix, and $\viota_n$ an $n$-vector of ones. For a matrix $\mX$ with full column rank, the matrix $\mP_{\vx} = \mX(\mX'\mX)^{-1}\mX'$ denotes the projection matrix of $\mX$. For a vector $\vx$ with elements $x_j$, $j = 1, \dots, p$, the subscripted vector $\vx_{\calJ}$ has $j$th element equal to $x_j$ if $j \in \calJ$ and 0, otherwise. Finally, the following standard notation is used
			\begin{align*}
			\|\vx\|_q
			= 
			\begin{cases}
			\sum_{j} 1_{\{x_j \neq 0\}}, &q = 0, \\
			(\sum_{j} |x_j|^q)^{\frac{1}{q}}, &0 < q < \infty.
			\end{cases}
			\end{align*}
		
	\section{From Quadratic to Conic statistics}\label{sec:dis}
		Let $\vmu$ be the $p$-vector of interest from the parameter space $\mathbb{R}^p$, and let $\vm$ be an estimator for $\vmu$ with covariance $\mSigma > 0$. The matrix $\mS \geq 0$ denotes a generic estimator of $\mSigma$. If $\mS$ is positive definite, a quadratic or Wald-type statistic that is typically used for testing the hypothesis $H_0 : \vmu = \vzeros$ is of the form
		\begin{align*}
			T = \sqrt{\vm'\mS^{-1}\vm},
		\end{align*}
		which is a weighted Euclidean norm of $\vm$. Let us slightly extend its definition so that $T = 0$ if $\vm = \vzeros$ and $\mS$ is singular. Then, $T$ can be equivalently written as a maximization problem. Specifically, let $\calS = \{\vlambda \in \mathbb{R}^p\ |\ \vlambda'\mS\vlambda = 1\}$ denote the unit ellipsoid induced by $\mS$, then
		
		\begin{align*}
			T
				= \max_{\vtheta'\vtheta = 1} \vm'\mS^{-\tfrac{1}{2}}\vtheta 
				= \max_{\vlambda\mS\vlambda = 1} \vm'\mS^{-\tfrac{1}{2}}\mS^{\tfrac{1}{2}}\vlambda
				= \max_{\vlambda \in \calS} \vm'\vlambda
				= \max_{\vlambda \in \calS^0} \vm'\vlambda,
		\end{align*}
		where the first step follows from the Cauchy-Schwarz inequality.\footnote{By the Cauchy-Schwarz inequality with $\vtheta \neq \vzeros$, we have $\vx'\vtheta\leq \sqrt{\vx'\vx}\sqrt{\vtheta'\vtheta}$, so that $\vx'\vtheta/\sqrt{\vtheta'\vtheta} \leq \sqrt{\vx'\vx}$, which implies $$\sqrt{\vx'\vx} = \max_{\substack{\vtheta}} \frac{\vx'\vtheta}{\sqrt{\vtheta'\vtheta}} = \max_{\substack{\vtheta \\ \vtheta'\vtheta = 1}} \vx'\vtheta.$$} Notice that the right-hand side reformulation does not use the inverse of $\mS$, but the equality holds as the maximization is unbounded if $\mS$ is singular and $\vm \neq \vzeros$.

		The key idea in this paper is to extend the definition of $T$ by restricting the weights $\vlambda$ to a section of $\calS$ that represents a direction of interest. The resulting statistic may exist even if $\mS$ is singular. Specifically, let $\calC \subseteq \mathbb{R}^p$ be a (not necessarily convex) cone and assume that $\calC$ is closed in $\mathbb{R}^p$. Then, I propose the following `conic' statistic
		\begin{align*}
			T_{\calC}
				= \max_{\substack{\vlambda \in \calS^0 \cap \calC}} \vm'\vlambda, 		
		\end{align*}
		which is an asymmetric semi-norm of $\vm$. The dependence on $\vm$ and $\calS$ is suppressed for notational convenience. Table \ref{tab:1} contains an overview of several test statistics that are nested by $T_{\calC}$ and can be obtained for different choices of $\calC$. Notice that $T = T_{\mathbb{R}^p}$.
			
		\begin{table}
			\caption{Several test statistics that are nested by the $T_{\calC}$ statistic.}
			\label{tab:1}
			\begin{tabular}{ll}
				\hline
				\hline
				$\calC$ &$T_{\calC}$ statistic\\
				\hline
				$\mathbb{R}^p$ & $\sqrt{\vm'\mS^{-1}\vm}$ (Wald statistic)\\
				$\calC_+ = \{\vlambda\ |\ \vlambda \geq \vzeros\}$ & $T_+$ \citep{koning2019exact} \\	 
				$\calC_{(1)} = \{\vlambda\ | \lambda_j = 0, j \neq 1\}$ &$|m_1| / s_1$ (two-sided $t$-statistic)\\
				$\calC_{(1)} \cap \calC_+$ &  $(m_1 / s_1)^+$ (one-sided $t$-statistic)\\
				$\calC_1 = \{\vlambda\ |\ \|\vlambda\|_0 \leq 1\}$ & $\max_j |m_j| / s_j$ \citep{bugni2016inference}\\
				$\calC_1 \cap \calC_+$  & $\max_j (m_j / s_j)^+$ \citep{chernozhukov2018inference} \\
				$\{(-1, 1, \vzeros), (1, -1, \vzeros)\}$ & $|m_1 - m_2| / \sqrt{s_1^2 + s_2^2}$, if $\mS$ is diagonal \citep{welch1947generalization} \\
				\hline
				\hline
			\end{tabular}
			\vspace{.1cm} \\
			\textit{This table contains some examples of statistics that are nested by the $T_{\calC}$ statistic. The left column describes the cone $\calC$ and the right column contains the corresponding $T_{\calC}$ statistic. The notation $(\cdot)^+ = \max\{\cdot, 0\}$ is used for the positive-part operator.}
		\end{table}
	
		\subsection{Existence of $T_{\calC}$}
			While the quadratic statistic $T_{\mathbb{R}^p}$ requires $\mS$ to be positive definite in order to exist, $T_{\calC}$ typically requires a weaker condition. Specifically, the following restricted eigenvalue condition (see e.g. \citealp{van2009conditions}) guarantees that $T_{\calC}$ exists.\footnote{The condition follows from first observing that the case $\vlambda = \vzeros$ is irrelevant, and then rewriting 
			\begin{align*}
				T_{\calC^\emptyset} 
					= \max_{\substack{\vlambda \in \calC^\emptyset}} \frac{\vm'\vlambda}{\sqrt{\vlambda'\mS\vlambda}} 
					= \max_{\substack{\vlambda \in \calC \\ \vm'\vlambda = 1}} \frac{1}{\sqrt{\vlambda'\mS\vlambda}} 
					= \frac{1}{\min_{\substack{\vlambda \in \calC \\ \vm'\vlambda = 1}} \sqrt{\vlambda'\mS\vlambda}}.
			\end{align*}}
			\begin{con}\label{con:exist}
				$\min_{\substack{\vlambda \in \calC \\ \vm'\vlambda = 1}} \vlambda'\mS\vlambda > 0$.
			\end{con}
			Notice this condition reduces to $\vlambda'\mS\vlambda > 0$ for all $\vlambda \in \mathbb{R}^p$, if $\calC = \mathbb{R}^p$ (and $\vm \neq \vzeros$). This is equivalent to $\mS > 0$, which indeed coincides with the condition for the typical Wald-type statistic $T_{\mathbb{R}^p}$ to exist.                                                               
		
	 	\subsection{Computation}		
			The following result shows that $T_{\calC}$ can be computed using restricted quadratic minimization. This is convenient as some of such problems have been well studied in the optimization literature. A proof of the result can be found in Appendix A.
			
			\begin{prp}\label{prp:quad}
				Let $\vlambda'\mS\vlambda > 0$ for all $\vlambda \in \calC^\emptyset$. Let $\widehat{\vbeta} = \argmin_{\vbeta \in \calC} 1 - 2\vm'\vbeta + \vbeta'\mS\vbeta$ and $\widehat{\vlambda} = \argmax_{\substack{\vlambda \in \calS^0 \cap\calC}} \vm'\vlambda$ be unique optimizers. Then
				\begin{align*}
					\widehat{\vlambda} = \frac{\widehat{\vbeta}}{\sqrt{\widehat{\vbeta}'\mS\widehat{\vbeta}}},
				\end{align*}
				if $\widehat{\vbeta} \neq \vzeros$ and $\widehat{\vlambda} = \vzeros$, otherwise.
			\end{prp}			
		
			In Section \ref{sec:sparse}, I discuss a case where the computation reduces to $\ell_0$- and $\ell_1$-regularized regression. Section \ref{sec:test:scone} covers the special case if $\mS$ is diagonal, where the computation simplifies for cones that are described by sign and sparsity restriction.

	 	\subsection{Geometric interpretation}	\label{sec:dis:geom}
			Condition \ref{con:exist} can be shown visually using the following geometric interpretation of $T_{\calC}$. Suppose that the vector $\widehat{\vlambda}$ with $\widehat{\vlambda}'\mS\widehat{\vlambda} \neq 0$ is the maximizing vector. Then $T_{\calC} = \vm'\widehat{\vlambda}$ can be viewed as the Euclidean norm of the vector $\vm$, after projection onto $\widehat{\vlambda}$ and scaling by the relative stretching of $\calS$ compared to the sphere:
			\begin{align*}
			\vm'\widehat{\vlambda}
			= \left\Vert\sqrt{\frac{\widehat{\vlambda}'\widehat{\vlambda}}{\widehat{\vlambda}'\mS\widehat{\vlambda}}}\mP_{\widehat{\vlambda}}\vm\right\Vert_2,
			\end{align*}
			where $\mP_{\widehat{\vlambda}} = \widehat{\vlambda}\widehat{\vlambda}'/\widehat{\vlambda}'\widehat{\vlambda}$ is the projection matrix of $\widehat{\vlambda}$.  A visual demonstration for both a singular and positive definite matrix $\mS$ is given in Figure \ref{fig:measure}. 
			
			In the left panel of Figure \ref{fig:measure}, we see that for the singular $\mS$, the ellipsoid $\calS$ is infinitely stretched along the scalar multiples $\gamma\vlambda_0$ of the vector $\vlambda_0 = (-1, 1)'$, for which $\gamma^2\vlambda_0'\mS\vlambda_0 = 0$. The chosen cone $\calC = \{\vlambda\ |\ \vlambda \geq \vzeros\}$ does not contain the positive or negative scalar multiples of $\vlambda_0$. So, $\vlambda'\mS\vlambda > 0$ for all $\vlambda \in \calC^\emptyset$. Therefore, Condition \ref{con:exist} is satisfied which means $T_{\calC}$ exists. The value of $T_{\calC}$ is represented as the length of the vector with the red diamond at its tip.
			
			In the right panel of Figure \ref{fig:measure}, $\mS$ is positive definite so that $\calS$ is a proper ellipse and Condition \ref{con:exist} is satisfied regardless of $\calC$. For the setting that is portrayed, it happens that the unconstrained maximizing vector $\widehat{\vlambda}$ are already contained in $\calC$. Therefore, $T_{\calC} = T_{\mathbb{R}^p}$. If $\widehat{\vlambda}$ was not contained in $\calC$, then $T_{\calC} \leq T_{\mathbb{R}^p}$.                                                                                               
			\begin{figure}
				\includegraphics[width = 7cm]{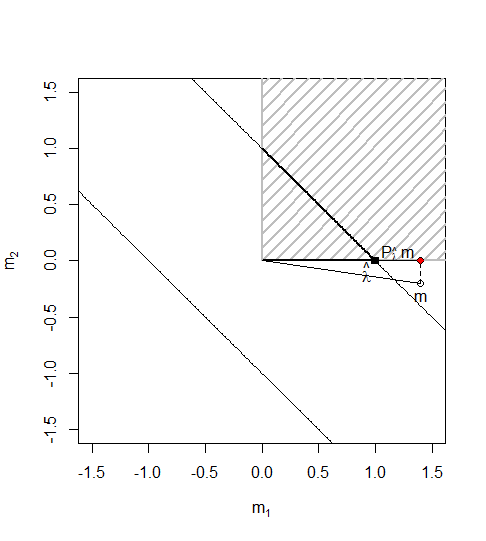}
				\includegraphics[width = 7cm]{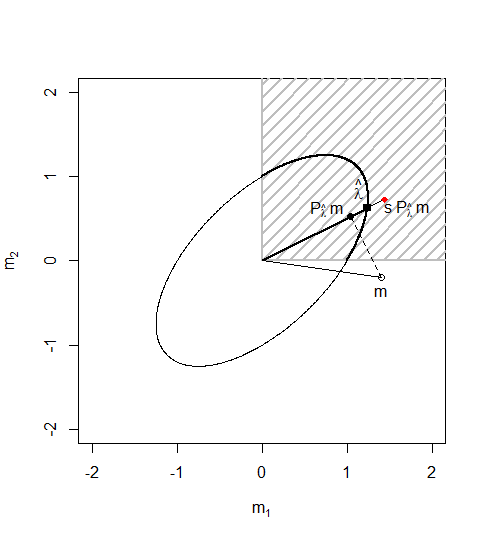}
				\caption{\textit{Visual demonstration of $T_{\calC}$ for a singular and positive-definite matrix $\mS$. In both plots, the diagonal elements of $\mS$ are 1. For the left plot, its off-diagonals are $1$, so that $\mS$ is singular and $\calS$ consists of two parallel lines. For the right plot, the off-diagonals are $-.6$, so that $\calS$ is a proper ellipse. The gray dashed area is the cone $\calC = \{\vlambda\ |\ \vlambda \geq \vzeros\}$. The thicker parts of $\calS$ represent the intersection $\calS \cap \calC$. In both plots, the thin line with an open circle at the end represents $\vm = (1.4, -.2)$, and the thicker line with the square at the end represents the maximizing vector $\widehat{\vlambda}$. The black solid dot is the projection $\mP_{\widehat{\vlambda}}\vm$, and the red solid diamond is the scaled projection $s\mP_{\widehat{\vlambda}}\vm$ of $\vm$ onto $\widehat{\vlambda}$, where $s = \sqrt{\frac{\widehat{\vlambda}'\widehat{\vlambda}}{\widehat{\vlambda}'\mS\widehat{\vlambda}}}$. The Euclidean norm of this scaled projection is $T_{\calC}$.}}
				\label{fig:measure}
			\end{figure}
			                                                                                                                                                            
	\section{Hypothesis testing}\label{sec:test}
		In this section, $T_{\calC}$ is used as a test statistic. In particular, I describe the hypotheses for which $T_{\calC}$ would be a suitable test statistic.
		
		\subsection{Null hypothesis}
			I start with describing the null hypothesis, for which the following notation is used. Let $T_\calC^* = \max_{\substack{\vlambda \in \calV \cap \calC}} \vmu'\vlambda$ denote the population equivalent of $T_\calC$, where $\calV^0 = \{\vy\ |\ \vy'\mSigma\vy = 1\}$. Let $\calC^\circ$ denote the polar cone of $\calC$, defined by $\calC^\circ = \{\vmu\ |\ \vmu'\vlambda \leq 0, \forall \vlambda \in \calC\}$. The polar cone of $\calC$ consists of the vectors that have a right angle or larger angle with any vector in $\calC$. It can therefore be interpreted as the cone that `points away' from $\calC$.
		
			Notice that the polar cone can equivalently be defined as $\calC^\circ = \{\vmu\ |\ \max_{\vlambda \in \calV^0\cap\calC} \vmu'\vlambda = 0\}$. So, the polar cone $\calC^\circ$ corresponds exactly to the parameter subspace where $T_{\calC}^* = 0$. Therefore, it seems appropriate to use $T_{\calC}$ as a test statistic for the null hypothesis
			\begin{align*}
				H_0 : \vmu \in \calC^{\circ}.
			\end{align*}
			
			Two examples are provided of different null hypotheses that can be constructed using different cones $\calC$.
			\begin{exm}
				The cone $\calC_+ = \{\vlambda\ |\ \vlambda \geq \vzeros\}$ has polar cone $\calC_+^\circ = \{\vmu \ |\ \vmu \leq \vzeros\}$, producing the multi-dimensional `one-sided' null hypothesis $H_0: \vmu \leq \vzeros$.
			\end{exm}
			
			\begin{exm}
				The cone $\calC_{(1)} = \{\vlambda\ |\ \lambda_j = 0,\ j \neq 1\}$  has polar cone $\calC_{(1)}^\circ = \{\vmu\ |\ \mu_1 = 0\}$ corresponding to the one-dimensional simple hypothesis $H_0 : \mu_1 = 0$. The intersection $\calC_{(1)+} = \calC_{(1)} \cap \calC_+$ has polar cone $\calC_{(1)+}^\circ$, which constructs the one-dimensional one-sided hypothesis $H_0: \mu_1 \leq 0$.
			\end{exm}	
		
		\subsection{Choosing a specific alternative hypothesis}			
			A key observation is that different cones $\calC$ can share the same polar cone, and therefore produce the same null  hypothesis. This is illustrated in the following example.
				
			\begin{exm}\label{ex:subalt}
				Let us consider $\calC_1 = \{\vlambda\ |\  \|\vlambda\|_0 \leq 1\}$, containing the cones that coincide with the axes, and the cone $\mathbb{R}^p$ that contains the entire parameter space. Notice that $T_{\calC_1}^* = \max_{\vlambda \in \calV^0 \cap \calC_1} \vmu'\vlambda = \max_j |\mu_j|/\sigma_j$ and $T_{\mathbb{R}^p}^* = \max_{\vlambda \in \calV^0} \vmu'\vlambda = \sqrt{\vmu'\mSigma^{-1}\vmu}$, are zero if and only if $\vmu = \vzeros$. Therefore, the polar cone corresponding to both $\calC_1$ and $\mathbb{R}^p$ is $\calC^\circ = \{\vzeros\}$. So, both $\calC_1$ and $\mathbb{R}^p$ produce the null hypothesis $H_0 : \vmu = \vzeros$. However, $T_{\calC_1}^*$ and $T_{\mathbb{R}^p}^*$ are clearly different for almost all $\vmu$ in $\mathbb{R}^p$: they only coincide if $\vmu \in \calC_1$.
			\end{exm}
		
			As different cones can correspond to the same null hypothesis, we have the freedom to change the cone $\calC$ without affecting the null hypothesis. This allows us to pick $\calC$ in such a way that power is directed towards a specific alternative of interest. To describe the alternative corresponding to a cone $\calC$, let us again consider $T_{\calC}^*$, the population equivalent of $T_{\calC}$. In particular, suppose that the maximizing vector is $\widehat{\vlambda} \neq \vzeros$, then
			
			\begin{align*}
				T_{\calC}^* 
				= \max_{\substack{\vlambda \in \calC \\ \vlambda'\mSigma\vlambda = 1}} \vmu'\vlambda
				= \max_{\substack{\mSigma^{-\frac{1}{2}}\vlambda \in \calC \\ \vlambda'\vlambda = 1}} \vmu'\mSigma^{-\frac{1}{2}}\vlambda
				= \max_{\substack{\vlambda \in \mSigma^{\frac{1}{2}}\calC \\ \vlambda'\vlambda = 1}} \vmu'\mSigma^{-\frac{1}{2}}\vlambda,
			\end{align*}
			where I use the notation $\mSigma^{\frac{1}{2}}\calC = \{\vx \in \mathbb{R}^p\ |\ \vx = \mSigma^{\frac{1}{2}}\vlambda,\ \vlambda \in \calC\}$. Notice that $T_{\calC}^*$ equals $T_{\mathbb{R}^p}^*$ if $\mSigma^{-\frac{1}{2}}\vmu \in \mSigma^{\frac{1}{2}}\calC^\emptyset$ by the Cauchy-Schwarz inequality, or equivalently if $\vmu \in \mSigma\calC^\emptyset$. If $\vmu \not\in \mSigma\calC^\emptyset$, then $T_{\calC}^* \leq T_{\mathbb{R}^p}^*$. So, $T_{\calC}^*$ is `large' if $\vmu \in \mSigma\calC^\emptyset$ and `small' if $\vmu \not\in \mSigma\calC^\emptyset$. Hence, $T_{\calC}$ seems to be a suitable test statistic for the alternative hypothesis
			\begin{align*}
				H_1^* : \vmu \in \mSigma\calC^\emptyset.
			\end{align*}
			
			Unlike the null hypothesis, this alternative hypothesis is not fully determined by $\calC$, but also depends on the the typically unknown parameter $\mSigma$. As a consequence, the exact parameter subspace towards which power is directed for a test based on $T_{\calC}$ may be unknown in practice. As this is inconvenient, I propose to instead use $T_{\calC}$ to test $H_0 : \vmu \in \calC^\circ$ against
			\begin{align*}
				H_1 : \vmu \in \calC^\emptyset.
			\end{align*}
			
			The consequence of this approach is a possible distortion of the alternative hypothesis. While this will clearly lead to a loss of power, the resulting test may still be substantially more powerful than a test based on a quadratic statistic. This is demonstrated numerically in the Monte Carlo experiments presented in Section \ref{sec:mc}. In addition, there exists an important special case where there is no distortion. This special case is treated in the following section.

		\subsection{Diagonal covariance matrices and scones}\label{sec:test:scone}
			There exists a type of cone $\calC$ and covariance matrix $\mSigma$ for which $\calC = \mSigma\calC$, so that $H_1$ and $H_1^*$ coincide. In particular, suppose that the elements of $\vm$ are uncorrelated, so that $\mSigma$ is diagonal. Let $\calD$ be a set of vectors that is closed under multiplication by a positive definite diagonal matrix. That is, if $\vlambda \in \calD$, then $\mD\vlambda \in \calD$, where $\mD > 0$. Notice here that $\calD$ is a cone.\footnote{This follows from the observation that multiplication by a scalar $\gamma > 0$ is equivalent to pre-multiplication with by diagonal matrix $\gamma\mI$, where $\mI$ is the identity matrix. So, $\calD$ is also closed under positive scalar multiplication, which makes it a cone.} So, we have indeed found a type of cone and covariance matrix for which $\calD = \mSigma\calD$. 
			
			
			I have been unable to find a name of these types of cones and will henceforth refer to them as \textit{scones} (sign cone). A scone can be viewed as a union of sub-orthants (i.e. a union of possibly lower dimensional orthants embedded in a higher dimensional space). As pre-multiplication by a diagonal matrix $\mD > 0$ is a sign-preserving operation, the sconic parameter subspaces relate directly to hypotheses concerning sign or sparsity restrictions. To provide some more intuitions for scones, the following examples are given.
			
			\begin{exm}
				The scones in $\mathbb{R}^2$ are the origin, the four half-axes, the four quadrants, and any union of any these objects.
			\end{exm}
		
			\begin{exm}\label{exm:scone:sparse}
				The set of $k$-sparse vectors, defined by $\calD_k = \{\vmu \ |\ \|\vmu\|_0 \leq k\}$ is a scone. To see this, suppose $\vmu = (\mu_1, \dots, \mu_k, 0, \dots, 0)'$, without loss of generality. So $\vmu \in \calD_k$. Let $\mD > 0$ have diagonal elements $d_1, \dots, d_p$. Then $\|\mD\vmu\|_0 = \|(d_1\mu_1, \dots, d_k\mu_k, 0, \dots, 0)'\|_0 = \|\vmu\|_0$. So $\mD\vmu \in \calD_k$.
			\end{exm}
		
			If $\mSigma$ is assumed to be diagonal, it would be sensible to use a diagonal estimator $\mS$. In this case, the computation of $T_{\calC}$ can be simplified substantially. In particular, Proposition \ref{prp:diag} shows that computing $T_{\calC}$ requires solving a minimum distance problem. Such problems are typically easier to solve than linear regression. For example, the following section considers testing against sparse subspaces, for which this problem has a closed-form solution.
			
			\begin{prp}\label{prp:diag}
				Let $\mS > 0$ be a diagonal matrix and let $\calD$ be a scone. Suppose that $\widehat{\vlambda} = \argmax_{\substack{\vlambda \in \calS^0\cap\calD}} \vm'\vlambda$ and $\widehat{\vbeta} = \argmin_{\substack{\vbeta \in \calD}} \|\mS^{-\tfrac{1}{2}}\vm- \vbeta\|_2^2$ are unique optimizers. Then
				\begin{align*}
				\widehat{\vlambda} 
				= \mS^{-\tfrac{1}{2}}\widehat{\vbeta}/\sqrt{\widehat{\vbeta}'\widehat{\vbeta}},
				\end{align*}
				if $\widehat{\vbeta} \neq \vzeros$ and $\widehat{\vlambda} = \vzeros$, otherwise.
			\end{prp}

	\section{Sparse parameter subspaces}\label{sec:sparse}
		In this section, I illustrate the methodology described in the previous sections by testing against \textit{sparse} parameter subspaces. For this illustration, I consider the following data generating process
		\begin{align}\label{eq:dgp}
			\mX = \viota_n\vmu' + \mE,
		\end{align}
		where $\mE$ is a random $n \times p$ matrix. 
		
		A sparse parameter subspace can be described as follows. For a given level of sparsity $k$, the $k$-sparse parameter subspace $\calD_{k} = \{\vmu \in \mathbb{R}^p \ |\  \|\vmu\|_0 \leq k\}$ consists of all $k$-sparse vectors of length $p$. Recall here that $\|\vx\|_0 = \sum_{j = 1}^{p} 1_{\{x_j \neq 0\}}$ is the $\ell_0$-norm of a vector $\vx$, which counts the number of non-zero elements in the vector. As detailed in Example \ref{exm:scone:sparse}, $\calD_{k}$ is a scone, because $\|\vx\|_0 = \|\mD\vx\|_0$ for all diagonal matrices $\mD > 0$. 
		
		The hypotheses for testing against a $k$-sparse parameter subspace can then be formulated as
		\begin{align*}
			H_0 &: \vmu = \vzeros, \\
			H_{1}^{k} &: \vmu \in \calD_k^\emptyset.
		\end{align*}
		The interpretation of the alternative $H_1^k$ is that (at most) $k$ elements of the equalities specified by the null hypothesis are violated.
		
		Following the methodology proposed in Section \ref{sec:test}, the statistic that we should use for the $k$-sparse alternative is $T_{\calD_k}$, which will be abbreviated as $T_{k}$ for notational convenience. In light of the data generating process \eqref{eq:dgp}, let us specify the statistic by choosing $\vm = \widehat{\vmu} := \tfrac{1}{n}\mX'\viota_n$ to be the sample mean and $\mS = \widehat{\mSigma} := \tfrac{1}{n} \mX'(\mI_n - \mP_{\viota_n})\mX$ to be the sample covariance matrix. 
		
		Following Condition \ref{con:exist}, a sufficient existence condition for $T_k$ is then given by
		\begin{con}\label{con:rank}
			$\rank(\widehat{\mSigma}) \geq k$.
		\end{con} 
		This condition is typically satisfied if the number of observations $n$ exceeds the hypothesized sparsity level $k$. So, if the hypothesized sparsity level is sufficiently small, is a much weaker condition than the requirement that $\widehat{\mSigma} > 0$ which requires that the number of observations exceeds the total number of parameters $p$.
			
		Although the maximizing vector $\widehat{\vlambda}$ in the computation of $T_k$ can be viewed as a by-product, it may also be of interest in its own right. In some applications it may be interesting to not just know \textit{whether} $H_0 : \vmu = \vzeros$ is violated, but also \textit{which} of the $p$ equalities specified by the null hypothesis are violated. A collection of $k$ potentially violated equalities is provided by the $k$ non-zero elements of $\widehat{\vlambda}$. The computation of $\widehat{\vlambda}$ is discussed in the following section.
		
		\subsection{Computation: full covariance matrix}\label{sec:sparse:full}		
			For the case that $\vm = \widehat{\vmu}$ and $\mS = \widehat{\mSigma}$, Proposition \ref{prp:regression} shows that the computation of $T_{\calC}$ reduces to regularized linear regression where the regularization is determined by $\calC$. The result is proven in Appendix A. 
			
			\begin{prp}\label{prp:regression}
				Let $\vlambda'\widehat{\mSigma}\vlambda > 0$ for all $\vlambda \in \calC^\emptyset$. Let $\widehat{\vbeta} = \argmin_{\substack{\vbeta \in \calC}} \frac{1}{n}\|\viota_n - \mX\vbeta\|_2^2$ and $\widehat{\vlambda} = \argmax_{\substack{\vlambda \in \calS^0 \cap\calC}} \vmu'\vlambda$ be unique optimizers. Then
				\begin{align*}
				\widehat{\vlambda}
				= \widehat{\vbeta}/\sqrt{\widehat{\vbeta}'\widehat{\mSigma}\widehat{\vbeta}},
				\end{align*}
				if $\widehat{\vbeta} \neq \vzeros$ and $\widehat{\vlambda} = \vzeros$, otherwise.
			\end{prp}
			If $\calC = \calD_k$, this result implies that $T_k$ can be computed using sparse linear regression or `best subset selection' (BSS), which is defined as follows
			\begin{align}\label{eq:bss}
				\widehat{\vbeta} = \argmin_{\substack{\vbeta \\ \|\vbeta\|_{0} \leq k}}\|\viota_n - \mX\vbeta\|_2^2.
			\end{align}
			While this optimization problem was long deemed computationally infeasible for $p \gtrsim 40$, \citet{bertsimas2016best} have recently shown that it can be solved for problems of practical size within reasonable time by using gradient descent methods and mixed-integer optimization solvers. In particular, they solve BSS with $n$ of order $10^3$ and $p$ of order $10^2$ in minutes. If $p$ and $n$ are very large or if re-sampling is used to construct a critical value, then this may still be costly in terms of computation to conduct a test. However, promising work by \citet{hazimeh2018fast} shows that problems of order $p = 10^6$ can be approximated in seconds. An alternative approach is to approximate the solution of \eqref{eq:bss} using greedy methods such as forward stepwise selection (see e.g. \citealt{hastie2009statistical, hastie2017extended}).
			
		\subsection{Computation: diagonal covariance matrix}\label{sec:sparse:diag}				
			Let us now also consider the computation of the statistic for the case that $\mS$ is a positive definite diagonal estimator. Popular examples of diagonal estimators are $\diag(\widehat{\mSigma})$ or the pooled variance estimator $\tr(\widehat{\mSigma}/p)\mI_p$. In order to distinguish the statistic for the diagonal covariance estimator, I use the superscripted notation $T_k^d$. Proposition \ref{prp:diag} then shows that if $\mS > 0$, $T_{k}^d$ can be computed by solving
			\begin{align*}
				\min_{\substack{\vbeta \\ \|\vbeta\|_{0} \leq k}} \|\mS^{-\frac{1}{2}}\widehat{\vmu} - \vbeta\|_2^2.
			\end{align*}
			It is straightforward to show that $\widehat{\vbeta} = \mS^{-1}\widehat{\vmu}_{\calJ_k}$, where $\calJ_k$ contains the largest $k$ elements of $\mS^{-1}\widehat{\vmu}$.\footnote{See e.g. Proposition 3 of \citet{bertsimas2016best} for a formal proof.} This leads to the closed form
			\begin{align*}
				T_{k}^d = \sqrt{\widehat{\vmu}'\mS^{-1}\widehat{\vmu}_{\calJ_k}}.
			\end{align*}
			If $\mS = \text{diag}(\widehat{\mSigma})$, the $T_{k}^d$  statistic is closely related to threshold-type statistics (see e.g. \citealp{fan1996test,zhong2013tests}), which include the screening statistic proposed by \citet{fan2015power}. They define the screening statistic as
			\begin{align*}
				J_0 = \sqrt{p}\widehat{\vmu}'\text{diag}(\widehat{\mSigma})^{-1}\widehat{\vmu}_{\calJ},
			\end{align*}
			where $\calJ = \{j \in \{1, \dots, p\}\ |\ |\widehat\mu_j| > \widehat\sigma_j \delta\}$, for a given threshold value $\delta$. Here, \citet{fan2015power} choose $\delta$ to grow sufficiently fast so that using 0 as critical value results in a test with asymptotic size 0. Notice that if $\delta$ is instead chosen such that $|\calJ| = k$, then $J_0$ is a monotone function of $T_{k}^d$: $J_0 = \sqrt{p}(T_{k}^d)^2$. Hence, the difference is that a threshold-type statistic implicitly defines the sparsity level, while the $T_k^d$ explicitly specifies the sparsity level of the alternative. The latter option seems more intuitive, but there may be applications for which a threshold-type specification of the sparsity level is desirable.

		\subsection{Lasso}\label{sec:sparse:lasso}
			Another popular tool to obtain sparse solutions is Lasso \citep{tibshirani1996regression}, which is linear regression under an $\ell_1$-norm restriction. Unlike $\calD_k$, the Lasso constraint set $\calB_t = \{\vlambda\ |\ \|\vlambda\|_1 \leq t\}$, where $t \geq 0$, is not a scone nor even a cone. However, it is possible to transform a non-cone restriction to a cone restriction, which I will demonstrate in this section.
			
			Let $\Lambda_t =  \calV \cap \calB_t$ be the intersection of the Lasso constraint set and $\calV$, for some $t \geq 0$. Then we can construct the cone $\calC_{\Lambda_t}= \{\vlambda\ |\ \gamma\vlambda \in \Lambda_t, \exists \gamma > 0\} \cup \{\vzeros\}$ that consists of all the scalar multiples of $\Lambda_t$ and the origin. This leads to the alternative
			\begin{align*}
				H_1 : \vmu \in \calC_{\Lambda_t}^\emptyset.
			\end{align*}
			Notice that if $\mSigma = \mI_p$, then this alternative coincides with the 1-sparse alternative when $t = 1$, and the Wald statistic if $t \geq \sqrt{p}$. A visual illustration of the two-dimensional case is given in Figure \ref{fig:L1}, for $t = 1.2$. This figure shows that the alternative corresponding to the Lasso constraint is set of cones that are centered around the axes, and whose angles are regulated by the parameter $t$. Therefore, the alternative induced by Lasso can be viewed as a \textit{near-sparse} alternative. The interpretation of near-sparse alternatives extends to higher dimensions.
		
			By Proposition \ref{prp:regression}, the computation of the test statistic $T_{\calC_{\Lambda_t}}$ reduces to Lasso regression
			\begin{align*}
				\widehat{\vbeta}
					= \argmin_{\substack{\vbeta \\ \|\vbeta\|_{1} \leq t/\sqrt{\widehat{\vbeta}'\widehat{\mSigma}\widehat{\vbeta}}}}\|\viota_n - \mX\vbeta\|_2^2,
			\end{align*}
			so that $\widehat{\vlambda} = \widehat{\vbeta} / \sqrt{\widehat{\vbeta}'\widehat{\mSigma}\widehat{\vbeta}}$.\footnote{Unfortunately, as of writing this manuscript, all existing efficient implementations of Lasso that I could find fail for a constant dependent variable. This has prohibited the inclusion of the Lasso based test in the Monte Carlo experiments. The authors of several prominent Lasso implementations have been made aware of this issue.}

			\begin{figure}
				\includegraphics[width = 12 cm]{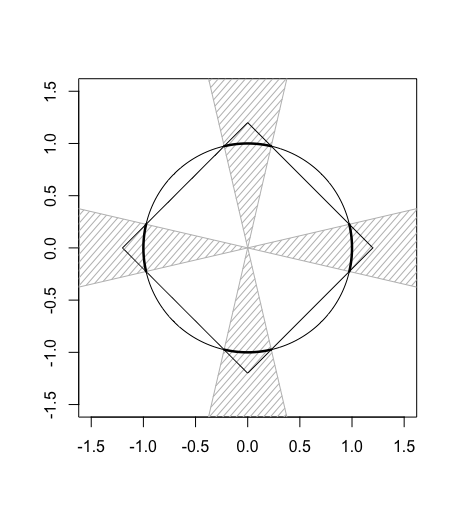}
				\caption{\textit{The construction of a Lasso sub-alternative in two dimensions. The area inside the diamond corresponds to the Lasso constraint set $\calB_{1.2}$, the circle is the unit sphere $\calU$, the thicker sections of the circle correspond to the intersection $\Lambda_{1.2} = \calU \cap \calB_{1.2}$, and the cone corresponding to the intersection $\calC_{\Lambda_{1.2}}^{\ell_1}$ is shown as the gray dashed areas.}}
				\label{fig:L1}
			\end{figure}
		
		\subsection{Critical values} \label{sec:sparse:critval}
			As the analytical or even the asymptotic derivation of a critical value for $T_{\calC}$ is non-trivial, I instead choose to compute a critical value using a re-sampling method. In particular, I use a reflection-based randomization test. Under a symmetry assumption on $\mE$, this leads to an exact critical value for tests with null hypothesis $H_0: \vmu = \vzeros$, so that size is controlled even in finite samples.\footnote{\citet{koning2019exact} apply a reflection-based randomization test to a composite null hypothesis. Although they are unable to formally show that such a test maintains size for all parameter values under the null, they do provide simulation results that support this conjecture.} If this symmetry assumption holds asymptotically, \citet{canay2017randomization} have shown that such a randomization test has asymptotic size control. The remainder of this section briefly describes the construction of such a critical value (see e.g. Ch. 15 of \citet{lehmann2006testing} for a more general discussion). 
			
			Let $\calR$ contain all $n \times n$ reflection matrices, which are diagonal matrices with diagonal elements in $\{-1, 1\}$. Then $\calR$ constitutes the reflection group of order $N = 2^n$. The assumption that permits the construction of an exact test is that $\mE$ and $\mR\mE$ share the same distribution, for all $\mR \in \calR$. This is equivalent to assuming that the distribution of $\mE$ is reflection symmetric.

			In order to perform a test, let $\calR^M := \{\mI_n, \mR_2, \dots, \mR_M\} \subseteq \calR$, where $\mR_i$ is drawn without replacement from $\calR \setminus \{\mI_n\}$, $i = 2, \dots, M$. Here, $M \leq N$ is the number of re-samples. Denote the set containing the reflection transformed data by $\calR^M_{\mX} = \{\mX, \mR_2\mX, \dots,  \mR _M\mX\}$. Let $T_{\calC}(\calR^M_{\mX} ) = \{T_{\calC}(\mX), T_{\calC}(\mR_2\mX), \dots, T_{\calC}(\mR_M\mX) \}$ be the set containing test statistic of interest computed on each of the reflection transformed data sets. Notice that under $H_0 : \vmu = \vzeros$ we have $\mX = \mE$, so that each of the elements of $T_{\calC}(\calR_{\mX}^M)$ follow the same distribution. So, under $H_0$, the probability that $T_{\calC}(\mX)$ is larger than the $1 - \alpha$ quantile $c_\alpha$ among the elements of $T_{\calC}(\calR_{\mX}^M)$ is $\alpha$. Therefore, a test with critical value $c_\alpha$ that rejects if $T_{\calC}(\mX) > c_\alpha$ has significance level $\alpha$.

	\section{Simulation experiments}\label{sec:mc} 
		In this section, the Monte Carlo simulation experiments are presented. In particular, I consider testing against sparse alternative, where the level of sparsity is varied. The goal of the experiments is to analyze the size and power properties of the $T_{\calC}$ test, and to compare it with the power enhancement approach of \citet{fan2015power}. For this purpose, I will vary the sample size, the number of parameters, the underlying covariance matrix and the level of sparsity in the parameter vector. Code to replicate the experiments will be made available at \url{https://github.com/nickwkoning}.
		
		\subsection{Data generation}
		
			I generate the $n \times p$ data matrix $\mX = \viota_n\vmu' + \mE\mA$, where the elements of $\mE$ are independently drawn from a standard normal distribution, and $\mA'\mA = \mSigma$, where $\mSigma$ has off-diagonal elements $\rho \in \{0,\ .5,\ .7\}$, $i \neq j$ and diagonal elements equal to 1. For the number of observations and parameters, I consider $n \in \{30,\ 250\}$ and $p \in \{100,\ 300,\ 500\}$,  respectively.
			
			The parameter vector of interest $\vmu$ has elements $\mu_j = b_{n, s} > 0$ for $1 \leq j \leq s$ and $\mu_j = 0$, if $s < j \leq p$. So, $\vmu$ is $s$-sparse. The values of $b_{n, s}$ are chosen to ensure that all tests that control size have rejection rates smaller than 1 at the significance level $\alpha = 0.05$. In particular, I use
				
			\begin{align*}
				b_{n, s}
					= 
					\begin{cases}
						0.75, \text{  if }\ n = 30,\ \ \, s = 1, \\
						0.25, \text{  if }\ n = 30,\ \ \, s = 20, \\
						0.25, \text{  if }\ n = 250,\ s = 1, \\
						0.07, \text{  if }\ n = 250,\ s = 20.
					\end{cases}
			\end{align*}
			
		\subsection{Tests}
			On the data described above, I test the null hypothesis $H_0 : \vmu = \vzeros$ against the $s$-sparse alternative $H_1^s : \vmu \in \calD_s$. As this sparsity level may not be  known in practice, I also consider a possible  over- and underspecification of $s$ by testing against $H_1^{1} : \vmu \in \calD_1$, if $s = 20$, and against $H_1^{20} : \vmu \in \calD_{20}$, if $s = 1$. 
			
			In particular, I use the $T_k$ statistic where $k \in \{1, 20\}$. As estimators for $\vmu$ and $\mSigma$, I choose the sample mean $\vm = \mX'\viota_n/n$, and sample covariance $\mS = \widehat{\mSigma} := \mX'(\mI_n - \mP_{\viota_n})\mX/n$. In addition, I also consider the $T_{k}^d$ statistic with $k \in \{1, 20\}$, where $\mS = \text{diag}(\widehat{\mSigma})$. Note here that the $T_1^d$ statistic and $T_1$ statistic coincide.
			
			The $T_k$ statistic is computed using the methodology proposed by \citet{hazimeh2018fast}. To construct a critical value for $T_k$ and $T_k^d$, I use reflection based randomization as described in Section \ref{sec:sparse:critval}. Note that this will yield an exact test, as the normal distribution is reflection symmetric. The number of reflection re-samples used is 1000. Performing a single test with the $T_{20}$ statistic on the largest setting with $n = 250$ and $p = 500$ takes approximately 45 seconds on a standard 2017 edition 13 inch Macbook Pro.
			
			For the comparison to the power enhancement technique, an initial test and power enhancement test must be selected. As initial test I use the standard Wald test, and as statistic for the enhancement test I use the same screening statistic as used by \citet{fan2015power}.\footnote{One may argue that a comparison could instead be made to the `feasible Wald test' proposed as initial test by \citet{fan2015power}. However, this test relies on an additional sparsity assumption on $\mSigma$ which would cloud the comparison.} This screening statistic is given by
			\begin{align*}
				J_0 = \sqrt{p}\vm'\diag(\widehat{\mSigma})^{-1}\vm_{\calJ},
			\end{align*}
			with $\calJ = \{j \in \{1, \dots, p\}\ |\ |m_j| > \widehat{\sigma}_{j}\delta_{n,p}\}$, where $\widehat{\sigma}_j^2$ are the diagonal elements of $\widehat{\mSigma}$, and $\delta_{n,p} = \log\log n \sqrt{\log p}$.\footnote{The code provided the online supplementary material of  \citet{fan2015power} suggests that  $1.06\delta_{n,p}$, $\sqrt{1.5}\delta_{n,p}$ and $.9\delta_{n,p}$ were used in their numerical experiments, instead of $\delta_{n,p}$. It is unclear where these constants come from, as they are not mentioned in the paper. I therefore choose to simply use $\delta_{n,p}$, instead.} As critical value for the Wald statistic I use $pn/(n - p)$ times the $1 - \alpha$ quantile of the $F$ distribution with $p$ and $n - p$ degrees of freedom (also known as a Hotelling $T^2$ test). If $p > n$, the initial test is not defined and I reject with probability $\alpha$. As critical value for the enhancement test I use 0. The combined test then rejects if either the initial test or enhancement test rejects.\footnote{The formulation of the power enhancement technique here differs slightly from the one described by \citet{fan2015power} and is of the form analyzed by \citet{kock2019power}. This formulation permits the use of the power enhancement approach even if the test statistic for the initial test does not exist.} The simulation experiments are also conducted on the initial test to infer the contribution of the constituent tests to the performance of the power enhancement technique.
			
			For each combination of setting and test, I perform 1000 repetitions and record the proportion of rejections. The results are reported in Tables \ref{tab:null}, \ref{tab:s1} and \ref{tab:s20}. 
			
		\subsection{Results}
			Table \ref{tab:null} reports the rejection rates of the tests on data that was generated under the null hypothesis. As these rejection rates are used to infer the size of the tests, they should ideally be close to the nominal significance level $\alpha = 0.05$. As expected, the table shows that the $T_{\calC}$ tests and the Wald test all have good size control under $H_0$. 
			
			In contrast, the power enhancement test suffers from a small size distortion in the large sample ($n = 250$). Furthermore, as the power enhancement test is based on an asymptotic critical value it fails to control size in the small sample ($n = 30$). The rejection rates for the power enhanced test for $n = 30$ will therefore be ignored in the remainder without mention. 
			
			Table \ref{tab:s1} contains the rejection rates under the alternative where $s = 1$, so that only a single element of $\vmu$ is unequal to zero. As this data is generated under the alternative, the rejection rates are used to infer the power of the test, which should be as large as possible. Overall, the $T_1$ test outperforms all other tests. For the large sample ($n = 250$), the power enhanced test performs second best and delivers a large power improvement compared to the Wald test. Despite the mis-specification of the sparsity level, the $T_{20}$ test performs reasonably well if $n = 250$, substantially outperforming the Wald test. The $T_{20}$ test performs especially well under large correlations. As expected, the $T_{20}^d$ performs better than the $T_{20}$ test under no correlations, but loses power as the correlations increase.
			
			Table \ref{tab:s20} presents the results for $s = 20$, which are more diverse than the results in Table \ref{tab:null} and \ref{tab:s1}. The $T_{20}^d$ test outperforms the other tests if $\rho = 0$ as the sparsity level is now correctly specified. If $\rho \neq 0$ and $n = 30$, the $T_{1}$ test outperforms the $T_{20}$ test even though $s = 20$. This may be a consequence of the low rank of $\widehat{\mSigma}$, so that Condition \ref{con:rank} is barely satisfied. For $n = 250$ and $\rho \neq 0$, the power enhancement test performs well in the low-dimensional setting $p = 100$. This performance seems almost entirely carried by the initial Wald test, which performs well as $s$ is large compared to $p$. In the settings where $p > n$, the $T_{20}$ test substantially outperforms the other tests.

			From the results described above, the following intuitions are drawn.
			\begin{enumerate}
				\setstretch{1.5}
				\item[$\bullet$] The $T_k^d$ test performs well if $\mSigma$ is diagonal and $k$ is close to $s$. It loses power as $k$ is further from $s$ or $\mSigma$ is further from diagonal.
				\item[$\bullet$] If $\mSigma$ is not diagonal, the $T_k$ test performs well if $k$ is close to the true sparsity level $s$, and $n$ is sufficiently large compared to $k$.
				\item[$\bullet$] The power enhancement test trades a small size distortion for a possibly large power gain.
				\item[$\bullet$] The power enhancement test is competitive with the $T_s$ test if $s$ is small, but loses power compared to the $T_s$ test if $s$ is somewhat larger.
			\end{enumerate}
	
		\begin{table}[ht]
			\centering
			\caption{Monte Carlo rejection rates if $H_0 : \vmu = \vzeros$ is true.}
			\begin{tabular}{rrrccc|ll}
				\hline
				\hline
				&&&\multicolumn{3}{c}{$T_{\calC}$} & \multicolumn{1}{c}{PE} & \multicolumn{1}{c}{Wald}\\
				\hline
				n & p & $\rho$ & $T_{1}$/$T_{1}^d$ & $T_{20}$  & $T_{20}^d$ &&\\
				\hline
				&		& 0   & .051 & .038  & .055 & .786 & .056\\ 
				& 100 	& .5  & .054 & .047 & .061 & .450 & .045\\ 
				&		& .7  & .045 & .048 & .054 & .270 & .059\\ 
				\hline
				&		& 0    & .041 & .043 & .053 & .895 & .043\\ 
				30 & 300 	& .5  & .056 & .055 & .046 & .434 & .049\\ 
				&		& .7  & .045 & .041 & .063 & .274 & .047\\ 
				\hline
				&		& 0    & .042 & .052 & .050 & .950 & .055\\ 
				&500	& .5  & .040 & .045 & .044 & .459 & .062\\ 
				&		& .7  & .056 & .043 & .048 & .261 & .049\\ 
				\hline
				\hline
				&   	& 0    & .051 & .054 & .042 & .080 & .058\\ 
				&100 	& .5  & .051 & .043 & .062 & .058 & .044\\ 
				&   	& .7  & .036 & .046 & .057 & .074 & .056\\ 
				\hline
				&		& 0    & .036 & .045 & .040 & .080 & .039\\ 
				250	&300	& .5  & .039 & .053 & .055 & .077 & .052 \\ 
				&		& .7  & .042 & .045 & .044 & .053 & .053\\ 
				\hline 
				&		& 0	   & .047  & .046 & .042 & .054 & .055\\ 
				&500	& .5  & .049 & .041 & .054 & .060 & .052\\ 
				&		& .7  & .047 & .049 & .058 & .060 & .049 \\ 
				\hline
				\hline
			\end{tabular} \\
			\vspace{.5cm}
			\begin{flushleft}
				\textit{The Monte Carlo rejection rates (between 0 and 1) for data generated under the null hypothesis $H_0 : \vmu = \vzeros$. Columns three to six contain the results for the reflection based $T_{\calC}$ tests. The column headed by PE contains the rejection rates for the power enhancement test, and the final column contains the results for the Wald test that functions as the initial test in the power enhancement test.}
			\end{flushleft}

			\label{tab:null}
		\end{table}

		\begin{table}[ht]
			\centering
			\caption{Monte Carlo rejection rates if $H_1^{1} : \vmu \in \calD_1$ is true.}
			\begin{tabular}{rrrccc|ll}
				\hline
				\hline
				&&&\multicolumn{3}{c}{$T_{\calC}$} & \multicolumn{1}{c}{PE} & \multicolumn{1}{c}{Wald}\\
				\hline
				n & p & $\rho$ & $T_{1}$/$T_{1}^d$ & $T_{20}$  & $T_{20}^d$ & &\\
				\hline
				&		& 0    & .\textbf{682} & .068 & .477 & .985  & .053\\ 
				& 100 	& .5  & .\textbf{775} & .128 & .085 & .972 & .059\\ 
				&		& .7  & .\textbf{842} & .176 & .081 & .959 & .049\\ 
				\hline
				&		& 0    & .\textbf{551} & .055 & .341 & .980  & .053 \\ 
				30 & 300 	& .5  & .\textbf{654} & .107 & .083 & .955 & .057\\ 
				&		& .7  & .\textbf{792} & .147 & .073 & .938 & .045\\ 
				\hline
				&		& 0    & .\textbf{498} & .064 & .305 & .991 &  .056\\ 
				&500	& .5  & .\textbf{630} & .091 & .085 & .949 & .052\\ 
				&		& .7  & .\textbf{724} & .147 & .080 & .917 & .051\\ 
				\hline
				\hline
				&   	& 0    & .\textbf{743} & .324  & .482  & .665 & .195\\ 
				&100 	& .5  & .\textbf{810} & .644 & .101 & .752 & .455\\ 
				&   	& .7  & .864 & .\textbf{921} & .096 & .876 & .755 \\ 
				\hline
				&		& 0    & .\textbf{637} & .257 & .352  & .506 & .054\\ 
				250	&300	& .5  & .\textbf{741} & .453 & .099 & .473  & .056\\ 
				&		& .7  & .\textbf{826} & .726 & .081 & .515 & .053\\ 
				\hline 
				&		& 0	   & .\textbf{601}  &  .205 & .312 & .429 & .061\\ 
				&500	& .5  & .\textbf{694} & .384 & .074 & .434 & .057\\ 
				&		& .7  & .\textbf{789} & .610 & .096   & .398 & .046\\ 
				\hline
				\hline
			\end{tabular} \\
			\vspace{.5cm}
			\begin{flushleft}
				\textit{The Monte Carlo rejection rates (between 0 and 1) for data generated under the alternative hypothesis with sparsity level $\|\vmu\|_0 = 1$. Columns three to six contain the results for the reflection based $T_{\calC}$ tests. The column headed by PE contains the rejection rates for the power enhancement test, and the final column contains the results for the Wald test that functions as the initial test in the power enhancement test.}
			\end{flushleft}
			\label{tab:s1}
		\end{table}
		\begin{table}[ht]
		\centering
		\caption{Monte Carlo rejection rates if $H_1^{20} : \vmu \in \calD_{20}$ is true.}
		\begin{tabular}{rrrccc|ll}
			\hline
			\hline
			&&&\multicolumn{3}{c}{$T_{\calC}$} & \multicolumn{1}{c}{PE} & \multicolumn{1}{c}{Wald}\\
			\hline
			n & p & $\rho$ & $T_{1}$/$T_{1}^d$ & $T_{20}$  & $T_{20}^d$ & &\\
			\hline
			&		& 0    & .374 & .102 & .\textbf{883} & .982 & .045\\ 
			& 100 	& .5  & .\textbf{258} & .118 & .217 & .796 & .053\\ 
			&		& .7  & .\textbf{280} & .162 & .201 & .567 & .048\\ 
			\hline
			&		& 0    & .205 & .060 & .\textbf{549} & .979  & .051 \\ 
			30 & 300 	& .5  & .\textbf{176} & .072 & .119 & .701 & .045\\ 
			&		& .7  & .\textbf{196} & .106 & .136 & .510 & .053\\ 
			\hline
			&		& 0    & .159 & .064 & .\textbf{409} & .980 & .060\\ 
			&500	& .5  & .\textbf{139} & .085 & .102 & .698 & .053\\ 
			&		& .7  & .\textbf{182} & .088 & .101 & .466 & .050\\ 
			\hline
			\hline
			&   	& 0    & .291 & .337  & .\textbf{752}  & .387 & .315\\ 
			&100 	& .5  & .203 & .537 & .130 & .\textbf{622} & .607\\ 
			&   	& .7  & .237 & .779 & .155 & .880 &.\textbf{898}\\ 
			\hline
			&		& 0    & .161 & .152 & .\textbf{387} & .094 & .052\\ 
			250	&300	& .5  & .141 & .\textbf{288} & .093 & .080  & .051\\ 
			&		& .7  & .175 & .\textbf{473} & .105 & .089 & .038\\ 
			\hline 
			&		& 0	   & .145 & .107 & .\textbf{287} & .088 & .041\\ 
			&500	& .5  & .123 & .\textbf{206} & .071 & .085 & .036\\ 
			&		& .7  & .138 & .\textbf{325} & .102 & .051 & .057\\ 
			\hline
			\hline
		\end{tabular} \\
		\vspace{.5cm}
		\begin{flushleft}
			\textit{The Monte Carlo rejection rates (between 0 and 1) for data generated under the alternative hypothesis with sparsity level $\|\vmu\|_0 = 20$. Columns three to six contain the results for the reflection based $T_{\calC}$ tests. The column headed by PE contains the rejection rates for the power enhancement test, and the final column contains the results for the Wald test that functions as the initial test in the power enhancement test.}
		\end{flushleft}
		\label{tab:s20}
	\end{table}

	\clearpage
	\bibliographystyle{abbrvnat}
	\bibliography{bibfile}	
	\clearpage
	\section{Appendix A}\label{app:A}
	\setcounter{prp}{0}
	
	In order to present the proofs of Proposition \ref{prp:quad}, I first provide the following simple lemma, where the notation $\gamma \calA$ is used to mean the set of containing the elements of $\calA \subseteq \mathbb{R}^p$ scalar multiplied by $\gamma$.
	\begin{lem}\label{lem:1}
		Let $\calC$ be a cone and $\gamma > 0$. If Condition \ref{con:exist} holds, then
		\begin{align*}
			\argmax_{\substack{\vlambda \in \calC \\ \vlambda'\mS\vlambda = \gamma^2}} 
				\vm'\vlambda
			=
			\gamma \argmax_{\substack{\vlambda \in \calC \\ \vlambda'\mS\vlambda = 1}} 
				\vm'\vlambda.
		\end{align*}
	\end{lem}
	\begin{proof}
		As Condition \ref{con:exist} holds, a maximizing argument exists. I find
		\begin{align*}
			\argmax_{\substack{\vlambda \in \calC \\ \vlambda'\mS\vlambda = \gamma^2}} 
					\vm'\vlambda
				&= \argmax_{\substack{\vlambda \in \calC \\ \tfrac{1}{\gamma^2}\vlambda'\mS\vlambda = 1}} 
					\vm'\vlambda
				= \argmax_{\substack{(\tfrac{1}{\gamma}\vlambda) \in \calC \\ (\tfrac{1}{\gamma}\vlambda)'\mS(\tfrac{1}{\gamma}\vlambda) = 1}} 
					\vm'\vlambda
				= \gamma\argmax_{\substack{\vlambda \in \calC \\ \vlambda'\mS\vlambda = 1}} 
					\vm'\vlambda. 
		\end{align*}
	\end{proof}

	\begin{prp}
		Let $\vlambda'\mS\vlambda > 0$ for all $\vlambda \in \calC^\emptyset$. Let $\widehat{\vbeta} = \argmin_{\vbeta \in \calC} 1 - 2\vm'\vbeta + \vbeta'\mS\vbeta$ and $\widehat{\vlambda} = \argmax_{\substack{\vlambda \in \calS^0\cap\calC}} \vm'\vlambda$ be unique optimizers. Then
		\begin{align*}
			\widehat{\vlambda} = \frac{\widehat{\vbeta}}{\sqrt{\widehat{\vbeta}'\mS\widehat{\vbeta}}},
		\end{align*}
		if $\widehat{\vbeta} \neq \vzeros$ and $\widehat{\vlambda} = \vzeros$, otherwise.
	\end{prp}
	\begin{proof}
		I will consider two cases: $\widehat{\vbeta} \neq \vzeros$ and $\widehat{\vbeta} = \vzeros$. \\
		
		\noindent \textbf{Case} $\widehat{\vbeta} \neq \vzeros$ \\
		I find
		\begin{align*}
			\widehat{\vlambda} 
				&=	\argmax_{
						 \substack{\vlambda \in \calC \\ 
									     \vlambda'\mS\vlambda = 1
						 }
					 } 
					 \vm'\vlambda 
				=  \tfrac{1}{\sqrt{\widehat{\vbeta}'\mSigma\widehat{\vbeta}}}
					\argmax_{
						\substack{\vlambda \in \calC \\ 
							\vlambda'\mS\vlambda = \widehat{\vbeta}'\mS\widehat{\vbeta}
						}
					} 
					\vm'\vlambda \\
				&=  \tfrac{1}{\sqrt{\widehat{\vbeta}'\mS\widehat{\vbeta}}}
					\argmin_{
						\substack{\vlambda \in \calC \\ 
							\vlambda'\mS\vlambda = \widehat{\vbeta}'\mS\widehat{\vbeta}
					 	}
					} 
					1 - 2\vm'\vlambda + \vlambda'\mS\vlambda
				=  \tfrac{1}{\sqrt{\widehat{\vbeta}'\mS\widehat{\vbeta}}}\widehat{\vbeta},
		\end{align*}
		where the second equality follows from Lemma \ref{lem:1}, the third equality from the fact that $\vlambda'\mS\vlambda$ is constant due to the constraint, and the final equality from the definition of $\widehat{\vbeta}$. \\

		\noindent \textbf{Case} $\widehat{\vbeta} = \vzeros$ \\
		I will prove that $\widehat{\vlambda} = \vzeros$ by contradiction. Suppose that $\widehat{\vlambda} \in \calC^\emptyset$ is arbitrarily given. Then
		$\widehat{\vlambda}'\mS\widehat{\vlambda} - 2\vm'\widehat{\vlambda} > 0$, as $\widehat{\vbeta} = \vzeros$ is a unique minimizer. So $\tfrac{1}{2} > \vm'\widehat{\vlambda} / \widehat{\vlambda}'\mS\widehat{\vlambda}$, because we assumed that $\vlambda'\mS\vlambda > 0$, for all $\vlambda \in \calC^{\emptyset}$. As $\widehat{\vlambda} \in \calC^\emptyset$, we have that $\gamma\widehat{\vlambda} \in \calC^\emptyset$ for all $\gamma > 0$. Then $\lim\limits_{\gamma \to 0^+}  \vm'(\gamma\widetilde{\vlambda}) / (\gamma\widetilde{\vlambda})'\mS(\gamma\widetilde{\vlambda}) = \lim\limits_{\gamma \to 0^+} \tfrac{1}{\gamma} \vm'\widehat{\vlambda} / \widehat{\vlambda}'\mS\widehat{\vlambda} < 1/2$, which implies $\vm'\widehat{\vlambda} \leq 0$. This contradicts the assumption that $\widehat{\vlambda}$ is a unique minimizer. Hence, $\widehat{\vlambda} = \vzeros$.
	\end{proof}

	\begin{prp}
		Let $\mS > 0$ be a diagonal matrix and let $\calD$ be a scone. Suppose that $\widehat{\vlambda} = \argmax_{\substack{\vlambda \in \calS^0\cap\calD}} \vm'\vlambda$ and $\widehat{\vbeta} = \argmin_{\substack{\vbeta \in \calD}} \|\mS^{-\tfrac{1}{2}}\vm- \vbeta\|_2^2$ are unique optimizers. Then
		\begin{align*}
		\widehat{\vlambda} 
		= \mS^{-\tfrac{1}{2}}\widehat{\vbeta}/\sqrt{\widehat{\vbeta}'\widehat{\vbeta}},
		\end{align*}
		if $\widehat{\vbeta} \neq \vzeros$ and $\widehat{\vlambda} = \vzeros$, otherwise.
	\end{prp}
	\begin{proof}
		I will consider two cases: $\widehat{\vbeta} \neq \vzeros$ and $\widehat{\vbeta} = \vzeros$. \\\\
		
		\noindent \textbf{Case} $\widehat{\vbeta} \neq \vzeros$ \\
		Using the substitution $\vbeta = \mS^{\frac{1}{2}}\vlambda$ and $\widehat{\vm} = \mS^{-\frac{1}{2}}\vm$ yields
		\begin{align*}
		\widehat{\vlambda} 
		= \argmax_{\substack{\vlambda \in \calD_{\mLambda}\\ \vlambda'\mS\vlambda = 1}} \widehat{\vmu}'\vlambda 
		= \mS^{-\frac{1}{2}}\argmax_{\substack{\vbeta \in \calD_{\mLambda} \\ \vbeta'\vbeta = 1}} \widehat{\vm}'\vbeta,
		\end{align*}
		where $\vbeta \in \calD_{\mLambda}$ if $\vlambda \in \calD_{\mLambda}$, as $\mS$ is diagonal and positive-definite. Define $\widehat{\vbeta} = \argmin_{\substack{\vbeta \in \calD_{\mLambda}}} \|\vbeta - \widehat{\vm}\|_2^2$. From Lemma 1 it follows that
		\begin{align*}
		\widehat{\vlambda}
		= \tfrac{1}{\sqrt{\widehat{\vbeta}'\widehat{\vbeta}}}\mS^{-\frac{1}{2}}\argmax_{\substack{\vbeta \in \calD_{\mLambda} \\ \vbeta'\vbeta = \widehat{\vbeta}'\widehat{\vbeta}}} \widehat{\vm}'\vbeta 
		= \tfrac{1}{\sqrt{\widehat{\vbeta}'\widehat{\vbeta}}}\mS^{-\frac{1}{2}}\argmin_{\substack{\vbeta \in \calD_{\mLambda}}} \|\vbeta - \widehat{\vm}\|_2^2
		= \tfrac{1}{\sqrt{\widehat{\vbeta}'\widehat{\vbeta}}}\mS^{-\frac{1}{2}}\widehat{\vbeta}.
		\end{align*}
		\noindent \textbf{Case} $\widehat{\vbeta} = \vzeros$. \\
		This case is analogous to the case that $\widehat{\vbeta} = \vzeros$ in the proof of Proposition \ref{prp:quad}.
	\end{proof}

	\begin{prp}
		Let $\vlambda'\widehat{\mSigma}\vlambda > 0$ for all $\vlambda \in \calC^\emptyset$. Let $\widehat{\vbeta} = \argmin_{\substack{\vbeta \in \calC}} \frac{1}{n}\|\viota_n - \mX\vbeta\|_2^2$ and $\widehat{\vlambda} = \argmax_{\substack{\vlambda \in (\calS\cap\calC)\cup\{\vzeros\}}} \vmu'\vlambda$ be unique optimizers. Then
		\begin{align*}
			\widehat{\vlambda}
				= \widehat{\vbeta}/\sqrt{\widehat{\vbeta}'\widehat{\mSigma}\widehat{\vbeta}},
		\end{align*}
		if $\widehat{\vbeta} \neq \vzeros$ and $\widehat{\vlambda} = \vzeros$, otherwise.
	\end{prp}
	\begin{proof}
		Let $\widehat{\mG} = \mX'\mX/n$ denote the Gramian matrix. I will consider two cases. \\
		
		\noindent \textbf{Case } $\widehat{\vbeta} \neq \vzeros$ \\
		Writing out the $\ell_2$-norm and using Proposition \ref{prp:quad} yields
		\begin{align*}
			\widehat{\vbeta} 
				= \argmin_{\substack{\vbeta \in \calC}} 1 - 2\widehat{\vmu}'\vbeta + \widehat{\vbeta}'\widehat{\mG}\widehat{\vbeta}
				= \sqrt{\widehat{\vbeta}'\widehat{\mG}\widehat{\vbeta}} \argmax_{\substack{\vlambda \in \calC \\ \vlambda'\widehat{\mG}\vlambda = 1}} \widehat{\vmu}'\vlambda,
		\end{align*}
		where I use the fact that $\vlambda'\widehat{\mSigma}\vlambda := \vlambda'\widehat{\mG}\vlambda - (\widehat{\vmu}'\vlambda)^2> 0$ for all $\vlambda \in \calC^{\emptyset}$, implies that $\vlambda'\widehat{\mG}\vlambda > 0$ for all $\vlambda \in \calC^{\emptyset}$. By Lemma \ref{lem:1}
		\begin{align*}
			\widehat{\vlambda} 
			= \argmax_{\substack{\vlambda \in \calC \\ \vlambda'\widehat{\mG}\vlambda - (\widehat{\vmu}'\vlambda)^2= 1}} \widehat{\vmu}'\vlambda
			= \argmax_{\substack{\vlambda \in \calC \\ \vlambda'\widehat{\mG}\vlambda = 1 + (\widehat{\vmu}'\widehat{\vlambda})^2}} \widehat{\vmu}'\vlambda
			= \sqrt{1 + (\widehat{\vmu}'\widehat{\vlambda})^2}\argmax_{\substack{\vlambda \in \calC \\ \vlambda'\widehat{\mG}\vlambda = 1}} \widehat{\vmu}'\vlambda.
		\end{align*}
		Combining this and applying some straightforward algebra yields
		\begin{align*}
			\widehat{\vlambda} 
				=  \sqrt{\tfrac{1 + (\widehat{\vmu}'\widehat{\vlambda})^2}{\widehat{\vbeta}'\widehat{\mG}\widehat{\vbeta}}} \widehat{\vbeta}
				=  \tfrac{1}{\sqrt{\widehat{\vbeta}'\widehat{\mSigma}\widehat{\vbeta}}} \widehat{\vbeta}.
		\end{align*}
		\noindent \textbf{Case} $\widehat{\vbeta} = \vzeros$ \\
		This case is analogous to the case that $\widehat{\vbeta} = \vzeros$ in the proof of Proposition \ref{prp:quad}.
	\end{proof}

\end{document}